\documentclass[12pt,twoside]{article}
\usepackage{amsmath,amsthm,amssymb,amscd,ascmac}
\usepackage{graphicx, mathrsfs}
\theoremstyle{definition}
\numberwithin{equation}{section}
\newtheorem{thm}{Theorem}[section]

\renewcommand{\det}{\mathop{\mathrm{det}}\,}

\theoremstyle{definition}

\newtheorem{rem}[thm]{Remark}

\usepackage{bm}

\begin{document}
% \title{Multivariate Bernoulli polynomials\\
% \vspace{1zh}
% \small{Dedicated to T. Oshima for his 70th birthday}}
% \author{Genki Shibukawa\thanks{
% This work was supported by Grant-in-Aid for JSPS Fellows (Number 18J00233).}\\}
% \date{
% \small MSC classes\,:\,11B68, 33C67, 43A90}
% \pagestyle{plain}
\title{Multivariate Bernoulli polynomials}
\author{Genki Shibukawa\thanks{
This work was supported by Grant-in-Aid for JSPS Fellows (Number 18J00233).}\\}
\date{
\small MSC classes\,:\,11B68, 33C67, 43A90}
\pagestyle{plain}

\maketitle

\begin{abstract}
We introduce a 
multivariate analogue of Bernoulli polynomials and give their fundamental properties: difference and differential relations, symmetry, explicit formula, inversion formula, multiplication theorem, and binomial type formula. 
Further, we consider a multivariate analogue of the multiple Bernoulli polynomials and give their fundamental properties. 
\end{abstract}

\section{Introduction}
The Bernoulli numbers $B_{m}$ are defined by the generating function
\begin{align}
\label{eq:def of B numbers}
\frac{u}{e^{u}-1}
   =
   \sum_{m=0}^{\infty}
   \frac{B_{m}}{m!}u^{m}, \quad |u|<2\pi , 
\end{align}
and the Bernoulli polynomials $B_{m}(z)$ by means of 
\begin{align}
\label{eq:def of B polynomials}
\frac{u}{e^{u}-1}e^{zu}
   =
   \sum_{m=0}^{\infty}
   \frac{B_{m}(z)}{m!}u^{m}, \quad |u|<2\pi . 
\end{align}

% Bernoulli polynomial $B_{m}(z)$ has the following fundamental properties (see for example Chapter\,1 Section\,13 \cite{E}). 
Bernoulli polynomial $B_{m}(z)$ has the following fundamental properties (see for example \cite{E} 1.13). 
\newpage 
\begin{align}
\label{eq:found prop1}
B_{m}(0)
   &=
   B_{m}, \\
\label{eq:found prop2}
B_{m}(z+1)-B_{m}(z)
   &=
   mz^{m-1} \quad (m\geq 0), \\
\label{eq:found prop3}
B_{m}^{\prime}(z)
   &=
   mB_{m-1}(z), \\
\label{eq:found prop4}
B_{m}(1-z)
   &=
   (-1)^{m}B_{m}(z), \\
\label{eq:found prop5}
B_{m}(z)
   &=
   \sum_{n=0}^{m}
   \binom{m}{n}B_{n}z^{m-n}, \\
\label{eq:found prop8}
z^{m}
   &=
   \frac{1}{m+1}\sum_{n=0}^{m}\binom{m+1}{n}B_{n}(z)
   =
   \sum_{n=0}^{m}\frac{1}{m-n+1}\binom{m}{n}B_{n}(z), \\  
\label{eq:found prop6}
\sum_{i=0}^{N-1}
   B_{m}\left(z+\frac{i}{N}\right)
   &=
   N^{1-m}B_{m}(Nz), \\
\label{eq:found prop7}
B_{m}(z+1)
   &=
   \sum_{n=0}^{m}
   \binom{m}{n}
   B_{n}\left(z\right).
\end{align}

% Let us describe proofs of (\ref{eq:found prop1}), (\ref{eq:found prop2}), (\ref{eq:found prop3}), (\ref{eq:found prop4}), (\ref{eq:found prop5}), (\ref{eq:found prop6}) and (\ref{eq:found prop7}). \\
Let us describe proofs of (\ref{eq:found prop1}) - (\ref{eq:found prop7}). \\
(\ref{eq:found prop1}) It follows from the definition of Bernoulli numbers (\ref{eq:def of B numbers}) and polynomials (\ref{eq:def of B polynomials}). \\
% For the difference equation (\ref{eq:found prop2}), by the generating function of Bernoulli polynomials (\ref{eq:def of B polynomials}) and the index law of exponential function $e^{zu}$, 
(\ref{eq:found prop2}) By the generating function of Bernoulli polynomials (\ref{eq:def of B polynomials}), the index law and the definition of exponential function $e^{zu}$, we have 
\begin{align}
\sum_{m\geq 0}
   (B_{m}(z+1)-B_{m}(z))\frac{u^{m}}{m!}
   &=
   \frac{u}{e^{u}-1}(e^{(z+1)u}-e^{zu}) \nonumber \\
   &=
   \frac{u}{e^{u}-1}(e^{u}e^{zu}-e^{zu}) \nonumber \\
   &=
   ue^{zu} \nonumber \\
   &=
   \sum_{m\geq 0}z^{m}u\frac{u^{m}}{m!} \nonumber \\
   &=
   \sum_{m\geq 0}z^{m}\frac{u^{m+1}}{(m+1)!}(m+1) \nonumber \\
   &=
   \sum_{m\geq 0}mz^{m-1}\frac{u^{m}}{m!}. \nonumber
\end{align} 
Similarly, we obtain the following. \\
(\ref{eq:found prop3}) 
\begin{align}
\sum_{m\geq 0}
   B_{m}^{\prime}(z)\frac{u^{m}}{m!}
   &=
   \frac{u}{e^{u}-1}\partial_{z}e^{zu} \nonumber \\
   &=
   \frac{u}{e^{u}-1}e^{zu}u \nonumber \\
   &=
   \sum_{m\geq 0}
   B_{m}(z)u\frac{u^{m}}{m!} \nonumber \\
   &=
   \sum_{m\geq 0}
   B_{m}(z)\frac{u^{m+1}}{(m+1)!}(m+1) \nonumber \\
   &=
   \sum_{m\geq 0}
   mB_{m-1}(z)\frac{u^{m}}{m!}. \nonumber
\end{align} 
(\ref{eq:found prop4}) 
\begin{align}
\sum_{m\geq 0}
   B_{m}(1-z)\frac{u^{m}}{m!}
   &=
   \frac{u}{e^{u}-1}e^{(1-z)u} \nonumber \\
   &=
   \frac{ue^{u}}{e^{u}-1}e^{-zu} \nonumber \\
   &=
   \frac{-u}{e^{-u}-1}e^{-zu} \nonumber \\
   &=
   \sum_{m\geq 0}
   B_{m}(z)\frac{(-u)^{m}}{m!} \nonumber \\
   &=
   \sum_{m\geq 0}
   (-1)^{m}B_{m}(z)\frac{u^{m}}{m!}. \nonumber
\end{align} 
(\ref{eq:found prop5})
\begin{align}
\sum_{m\geq 0}
   B_{m}(z)\frac{u^{m}}{m!}
   &=
   \frac{u}{e^{u}-1}e^{zu} \nonumber \\
   &=
   \sum_{N\geq 0}B_{N}\frac{u^{N}}{N!}\sum_{n\geq 0}z^{n}\frac{u^{n}}{n!} \nonumber \\
   &=
   \sum_{N\geq 0}B_{N}\sum_{n\geq 0}z^{n}\frac{u^{N}}{N!}\frac{u^{n}}{n!} \nonumber \\
   &=
   \sum_{m\geq 0}\sum_{n= 0}^{m}B_{m-n}z^{n}\binom{m}{n}\frac{u^{m}}{m!}. \nonumber
\end{align}
(\ref{eq:found prop8})
\begin{align}
\sum_{m\geq 0}
   z^{m}\frac{u^{m}}{m!}
   &=
   e^{zu} \nonumber \\
   &=
   \frac{e^{u}-1}{u}
   \sum_{n\geq 0}
   B_{n}(z)\frac{u^{n}}{n!} \nonumber \\
   &=
   \sum_{N\geq 0}\frac{1}{N+1}\frac{u^{N}}{N!}
   \sum_{n\geq 0}
   B_{n}(z)\frac{u^{n}}{n!} \nonumber \\
   &=
   \sum_{N\geq 0}\sum_{n\geq 0}
   \frac{1}{N+1}
   B_{n}(z)
   \frac{u^{N}}{N!}
   \frac{u^{n}}{n!} \nonumber \\
   &=
   \sum_{N\geq 0}\sum_{n\geq 0}
   \frac{1}{N+1}
   B_{n}(z)
   \sum_{m\geq 0}
   \binom{m}{N}\frac{u^{m}}{m!} \nonumber \\
   &=
   \sum_{m\geq 0}
   \sum_{n=0}^{m}
   \frac{1}{m-n+1}
   \binom{m}{n}
   B_{n}(z)
   \frac{u^{m}}{m!}. \nonumber
\end{align}
(\ref{eq:found prop6})
\begin{align}
\sum_{m\geq 0}
   \sum_{i=0}^{N-1}
   B_{m}\left(z+\frac{i}{N}\right)
   \frac{u^{m}}{m!}
   &=
   \sum_{i=0}^{N-1}
   \frac{u}{e^{u}-1}
   e^{\left(z+\frac{i}{N}\right)u} \nonumber \\
   &=
   \frac{u}{e^{u}-1}
   e^{zu}
   \sum_{i=0}^{N-1}e^{\frac{i}{N}u} \nonumber \\
   &=
   \frac{u}{e^{u}-1}
   e^{zu}
   \frac{e^{u}-1}{e^{\frac{u}{N}}-1} \nonumber \\
   &=
   N\frac{\frac{u}{N}}{e^{\frac{u}{N}}-1}
   e^{Nz\frac{u}{N}} \nonumber \\
   &=
   N\sum_{m\geq 0}B_{m}(Nz)\frac{1}{m!}\left(\frac{u}{N}\right)^{m} \nonumber \\
   &=
   \sum_{m\geq 0}N^{1-m}B_{m}(Nz)\frac{u^{m}}{m!}. \nonumber
\end{align}
 (\ref{eq:found prop7})
\begin{align}
\sum_{m\geq 0}
   B_{m}(z+1)
   \frac{u^{m}}{m!}
   &=
   \frac{u}{e^{u}-1}
   e^{(z+1)u} \nonumber \\
   &=
   e^{u}\frac{u}{e^{u}-1}
   e^{zu} \nonumber \\
   &=
   \sum_{n\geq 0}B_{n}(z)e^{u}\frac{u^{n}}{n!} \nonumber \\
   &=
   \sum_{n\geq 0}B_{n}(z)
   \sum_{m\geq 0}
   \binom{m}{n}\frac{u^{m}}{m!} \nonumber \\
   &=
   \sum_{m\geq 0}
   \sum_{n\geq 0}
   \binom{m}{n}
   B_{n}(z)
   \frac{u^{m}}{m!}. \nonumber   
\end{align}

We remark that in the above proofs we only use the following formulas which are trivial results in the one variable case. \\
\noindent
\underline{{\bf{Pieri type formulas}}} 
For any nonnegative integer $m \in \mathbb{Z}$, 
\begin{align}
\label{eq:0key formula 4}
e^{u}\frac{u^{m}}{m!}
   &=
   \sum_{n\geq 0}
   \binom{n}{m}\frac{u^{n}}{n!}.
\end{align}
Since
$$
e^{u}=\sum_{N\geq 0}\frac{1}{N!}u^{N}
$$
and comparing the terms of degree $N+m$ in (\ref{eq:0key formula 4}), we have 
\begin{align}
\label{eq:0key formula 5}
\frac{u^{N}}{N!}\frac{u^{m}}{m!}
   &=
   \binom{m+N}{m}\frac{u^{m+N}}{(m+N)!}.
\end{align}
In particular, the $N=1$ case of (\ref{eq:0key formula 5}) is the following :  
\begin{align}
\label{eq:0Pieri 2}
u\frac{u^{m}}{m!}
   =
   \binom{m+1}{m}\frac{u^{m+1}}{(m+1)!}
   =
   (m+1)\frac{u^{m+1}}{(m+1)!}.
\end{align}
\underline{{\bf{Properties of $e^{zu}$}}} 
\begin{align}
\label{eq:0key formula 3}
\partial_{z}e^{zu}
   &=
   e^{zu}u. 
\end{align}
In particular, we obtain the index law of $e^{zu}$
\begin{align}
\label{eq:0key formula 2}
e^{(1+z)u}
   &=
   e^{\partial_{z}}
   e^{zu}
   =
   e^{zu}e^{u}.
\end{align}
\noindent
\underline{{\bf{Other formula (trivial!)}}} For any nonnegative integer $N \in \mathbb{Z}$, 
\begin{align}
\label{eq:0key formula 1}
u^{N}
   &=
   N!\frac{u^{N}}{N!}. 
\end{align}

On the other hand, a multivariate analogue of the formulas (\ref{eq:0key formula 4}) - (\ref{eq:0key formula 1}) has been studied (see Section\,2), which is non-trivial results unlike the one variable case. 
Therefore if we give a good multivariate analogue of Bernoulli polynomials which can be applied a multivariate analogue of (\ref{eq:0key formula 4}) - (\ref{eq:0key formula 1}), then we drive a multivariate analogue of (\ref{eq:found prop1}) - (\ref{eq:found prop7}).

In this article, we introduce a multivariate analogue of Bernoulli polynomials $B_{m}(z)$ by Jack polynomials and others, which we call ``{\it{multivariate Bernoulli polynomials}}''. 
We also provide a multivariate analogue of (\ref{eq:found prop1}) - (\ref{eq:found prop7}) based on a multivariate analogue of (\ref{eq:0key formula 4}) - (\ref{eq:0key formula 1}). 
Further, we consider a multivariate analogue of the multiple Bernoulli polynomials and give their fundamental properties.

The content of this article is as follows. 
In Section\,2, we introduce a multivariate analysis which is a natural generalization of special functions for matrix arguments. 
In particular, we explain a multivariate analogue of (\ref{eq:0key formula 4}) - (\ref{eq:0key formula 1}). 
% and mention some special cases of these formulas. 
Section\,3 is the main part of this article. 
In this section, we introduce multivariate Bernoulli polynomials by a generating function which is a natural multivariate analogue of (\ref{eq:def of B polynomials}), and give their fundamental properties. 
We also investigate a multivariate analogue of the multiple Bernoulli polynomials which is a multiple analogue of our multivariate Bernoulli polynomials in Section\,4.

\section{Preliminaries}
Refer to \cite{Ka}, \cite{Ko}, \cite{L}, \cite{M}, \cite{S}, \cite{VK} for the details in this section. 
Let $r \in \mathbb{Z}_{\geq 1}$, $d\in \mathbb{C}$ and  
\begin{align}
\mathcal{P}
   &:=
   \{\mathbf{m}=(m_{1},\ldots,m_{r}) \in \mathbb{Z}^{r} \mid m_{1}\geq \cdots \geq m_{r} \geq 0\}, \nonumber \\
\delta 
   &:=
   (r-1,r-2,\ldots ,2,1,0) \in \mathcal{P}, \nonumber \\
% e_{r,k}(\mathbf{z})
%    &:=
%    \sum_{1\leq i_{1}<\cdots <i_{k}\leq r}
%       z_{i_{1}}\cdots z_{i_{k}} \quad (k=1,\ldots ,r), 
%    \quad e_{r,0}(\mathbf{z}):=1, \quad 
|\mathbf{z}|&:=z_{1}+\cdots +z_{r}, \nonumber \\
E_{k}(\mathbf{z})
   &:=
   \sum_{j=1}^{r}
   z_{j}^{k}\partial_{z_{j}} \quad (k \in \mathbb{Z}_{\geq 0}), \nonumber \\
D_{k}(\mathbf{z})
   &:=
   \sum_{j=1}^{r}
   z_{j}^{k}\partial_{z_{j}}^{2}
      +d 
      \sum_{1\leq j\not= l \leq r}
         \frac{z_{j}^{k}}{z_{j}-z_{l}}
         \partial_{z_{j}}
 \quad (k \in \mathbb{Z}_{\geq 0}). \nonumber
\end{align}
For any partition $\mathbf{m}=(m_{1},\ldots,m_{r}) \in \mathcal{P}$ and $\mathbf{z}=(z_{1},\ldots ,z_{r}) \in \mathbb{C}^{r}$, put 
$$
m_{\mathbf{m}}(\mathbf{z})
   :=
   \sum_{\mathbf{n} \in \mathfrak{S}_{r}\mathbf{m}}z^{\mathbf{n}}, 
$$
where $\mathfrak{S}_{r}$ is the symmetric group in $r$ letters and $\mathbf{z}^{\mathbf{n}}:=z_{1}^{n_{1}}\cdots z_{r}^{n_{r}}$. 
We define {\bf{Jack polynomials}} $P_{\mathbf{m}}\left(\mathbf{z};\frac{d}{2}\right)$ by the following two conditions.
\begin{align}
{\rm{(1)}} \,& 
D_{2}(\mathbf{z})P_{\mathbf{m}}\left(\mathbf{z};\frac{d}{2}\right)
   =
   P_{\mathbf{m}}\left(\mathbf{z};\frac{d}{2}\right)
   \sum_{j=1}^{r}m_{j}\left(m_{j}-1-d(r-j)\right), \nonumber \\
{\rm{(2)}} \,& 
P_{\mathbf{m}}\left(\mathbf{z};\frac{d}{2}\right)
   =
   m_{\mathbf{m}}(\mathbf{z})
   +\sum_{\mathbf{k}<\mathbf{m}}c_{\mathbf{m}\mathbf{k}}m_{\mathbf{k}}(\mathbf{z}). \nonumber 
\end{align}
Here, $<$ is the dominance partial ordering which is defined by 
$$
\mathbf{k}<\mathbf{m} \quad \Leftrightarrow \quad \mathbf{k}\not=\mathbf{m}, \quad 
\sum_{l=1}^{i}k_{l}\leq \sum_{l=1}^{i}m_{l} \quad i=1,\ldots, r.
$$
Similarly, {\bf{the shifted (or interpolation) Jack polynomials}} $P_{\mathbf{m}}^{\mathrm{ip}}\left(\mathbf{z};\frac{d}{2}\right)$
are defined by the following two conditions.
\begin{align}
{\rm{(1)}{}^{\mathrm{ip}}} \, &  
P_{\mathbf{k}}^{\mathrm{ip}}\left(\mathbf{m}+\frac{d}{2}\delta ;\frac{d}{2}\right)=0, \quad \text{unless $\mathbf{k} \subset \mathbf{m} \in \mathcal{P}$} \nonumber \\
{\rm{(2)}{}^{\mathrm{ip}}} \, &
P_{\mathbf{m}}^{\mathrm{ip}}\left(\mathbf{z};\frac{d}{2}\right)
   =
   P_{\mathbf{m}}\left(\mathbf{z};\frac{d}{2}\right)
   +\text{(lower terms)}. \nonumber
\end{align}
Further, we put 
% \begin{align}
% \Phi _{\mathbf{m}}^{(d)}(\mathbf{z})
%    &:=
%    \frac{P_{\mathbf{m}}\left(\mathbf{z};\frac{d}{2}\right)}{P_{\mathbf{m}}\left(\mathbf{1};\frac{d}{2}\right)}\,\, \text{(normalized Jack polynomials)},  \nonumber \\
% \Psi _{\mathbf{m}}^{(d)}(\mathbf{z})
%    &:=
%    \frac{P_{\mathbf{m}}\left(\mathbf{z};\frac{d}{2}\right)}{P_{\mathbf{m}}^{\mathrm{ip}}\left(\mathbf{m}+\frac{d}{2}\delta ;\frac{d}{2}\right)}, \nonumber \\
% \binom{\mathbf{z}}{\mathbf{k}}^{(d)}
%    &:=
%    \frac{P_{\mathbf{k}}^{\mathrm{ip}}\left(\mathbf{z}+\frac{d}{2}\delta ;\frac{d}{2}\right)}{P_{\mathbf{k}}^{\mathrm{ip}}\left(\mathbf{k}+\frac{d}{2}\delta ;\frac{d}{2}\right)} \,\, \text{(generalized (or Jack) binomial coefficients)}, \nonumber \\
% {_{0}\mathcal{F}_0}^{(d)}\left( \begin{matrix} \\ \end{matrix};\mathbf{z},\mathbf{w}\right)
%    &:=
%    \sum_{\mathbf{m} \in \mathcal{P}}
%    \Psi_{\mathbf{k}}^{(d)}(\mathbf{z})\Phi_{\mathbf{k}}^{(d)}(\mathbf{w}). \nonumber 
% \end{align}
\begin{align}
\Phi _{\mathbf{m}}^{(d)}(\mathbf{z})
   &:=
   \frac{P_{\mathbf{m}}\left(\mathbf{z};\frac{d}{2}\right)}{P_{\mathbf{m}}\left(\mathbf{1};\frac{d}{2}\right)}\,\, \text{(normalized Jack polynomials)},  \nonumber \\
\Psi _{\mathbf{m}}^{(d)}(\mathbf{z})
   &:=
   \frac{P_{\mathbf{m}}\left(\mathbf{z};\frac{d}{2}\right)}{P_{\mathbf{m}}^{\mathrm{ip}}\left(\mathbf{m}+\frac{d}{2}\delta ;\frac{d}{2}\right)} \nonumber
\end{align}
and 
\begin{align}
\binom{\mathbf{z}}{\mathbf{k}}^{(d)}
   &:=
   \frac{P_{\mathbf{k}}^{\mathrm{ip}}\left(\mathbf{z}+\frac{d}{2}\delta ;\frac{d}{2}\right)}{P_{\mathbf{k}}^{\mathrm{ip}}\left(\mathbf{k}+\frac{d}{2}\delta ;\frac{d}{2}\right)} \,\, \text{(generalized (or Jack) binomial coefficients)}, \nonumber \\
{_{0}\mathcal{F}_0}^{(d)}\left( \begin{matrix} \\ \end{matrix};\mathbf{z},\mathbf{u}\right)
   &:=
   \sum_{\mathbf{m} \in \mathcal{P}}
   \Psi_{\mathbf{m}}^{(d)}(\mathbf{z})\Phi_{\mathbf{m}}^{(d)}(\mathbf{u})
   =
   \sum_{\mathbf{m} \in \mathcal{P}}
   \Phi_{\mathbf{m}}^{(d)}(\mathbf{z})\Psi_{\mathbf{m}}^{(d)}(\mathbf{u}). \nonumber 
\end{align}
\noindent
\underline{{\bf{Special values}}} From \cite{M} V\!I (6.14), (10.20) and \cite{Ko} (4.8), we have 
\begin{align}
\label{eq:special value1}
P_{\mathbf{m}}\left(\mathbf{1};\frac{d}{2}\right)
   &=
   \prod_{(i,j) \in \mathbf{m}}
      \frac{j-1+\frac{d}{2}(r-i+1)}{m_{i}-j+\frac{d}{2}(m_{j}^{\prime}-i+1)}
   =
   \prod_{1\leq i<j\leq r}
         \frac{\left(\frac{d}{2}(j-i+1)\right)_{m_{i}-m_{j}}}{\left(\frac{d}{2}(j-i)\right)_{m_{i}-m_{j}}}.
\end{align}
Further, by \cite{Ko} (7.4) and (7.5)
\begin{align}
P_{\mathbf{m}}^{\mathrm{ip}}\left(\mathbf{m}+\frac{d}{2}\delta ;\frac{d}{2}\right)
   &=
   \prod_{(i,j) \in \mathbf{m}}
      \left(m_{i}-j+1+\frac{d}{2}(m_{j}^{\prime}-i)\right) \nonumber \\
\label{eq:special value2}
   &=
   \prod_{j=1}^{r}
   \left(\frac{d}{2}(r-j)+1\right)_{m_{j}}
   \prod_{1\leq i<j\leq r}
         \frac{\left(\frac{d}{2}(j-i-1)+1\right)_{m_{i}-m_{j}}}{\left(\frac{d}{2}(j-i)+1\right)_{m_{i}-m_{j}}}. 
\end{align}
Although these multivariate special functions are very complicated, we write down these functions explicitly in $r=1$, $r=2$ and $d=2$. \\
{\bf{The $r=1$ case}} For non positive integer $m$ and $z \in \mathbb{C}$, 
\begin{align}
P_{m}\left(z;\frac{d}{2}\right)
   &=
   z^{m}, \quad
P_{m}^{\mathrm{ip}}\left(z ;\frac{d}{2}\right)
   =
   \begin{cases}
   z(z-1)\cdots (z-m+1) & (m\not=0) \\
   1 & (m=0)
   \end{cases}. \nonumber       
\end{align}
Further, 
\begin{align}
P_{m}\left(1;\frac{d}{2}\right)
   &=
   1, \quad 
P_{m}^{\mathrm{ip}}\left(m ;\frac{d}{2}\right)
   =
   m!, \quad 
\Phi _{m}^{(d)}(z)
   =
   z^{m}, \quad 
\Psi _{m}^{(d)}(z)
   =
   \frac{z^{m}}{m!}, \nonumber \\
\binom{z}{k}^{(d)}
   &=
   \begin{cases}
   \frac{z(z-1)\cdots (z-k+1)}{k!} & (k\not=0) \\
   1 & (k=0)
   \end{cases}, \nonumber \\
{_{0}\mathcal{F}_0}^{(d)}\left( \begin{matrix} \\ \end{matrix};z,u\right)
   &=
   \sum_{m\geq 0}
   \frac{z^{m}}{m!}u^{m}
   =
   \sum_{m\geq 0}
   z^{m}\frac{u^{m}}{m!}
   =e^{zu}. \nonumber 
\end{align}
{\bf{The $r=2$ case}} (see \cite{Ko} 10.3, \cite{VK} 3.2.1) For any partition $\mathbf{m}=(m_{1},m_{2}) \in \mathcal{P}$ and $\mathbf{z}=(z_{1} ,z_{2}) \in \mathbb{C}^{2}$, 
\begin{align}
P_{\mathbf{m}}\left(\mathbf{z};\frac{d}{2}\right)
   &=
   z_{1}^{m_{1}}z_{2}^{m_{2}}
   {_{2}{F}_1}\left( \begin{matrix} -m_{1}+m_{2},\frac{d}{2} \\ 1-m_{1}+m_{2}-\frac{d}{2} \end{matrix};\frac{z_{2}}{z_{1}}\right) \nonumber \\
P_{\mathbf{m}}^{\mathrm{ip}}\left(\mathbf{z} ;\frac{d}{2}\right)
   &=
   (-1)^{m_{1}+m_{2}}
   (-z_{1})_{m_{2}}(-z_{2})_{m_{1}}
   {_{3}{F}_2}\left( \begin{matrix} -m_{1}+m_{2},\frac{d}{2}, -m_{1}+1-\frac{d}{2}+z_{1}\\ 1-m_{1}+m_{2}-\frac{d}{2},-m_{1}+1+z_{2} \end{matrix};1\right). \nonumber      
\end{align}
Further, 
\begin{align}
P_{\mathbf{m}}\left(\mathbf{1};\frac{d}{2}\right)
   &=
   \frac{(d)_{m_{1}-m_{2}}}{\left(\frac{d}{2}\right)_{m_{1}-m_{2}}} \nonumber \\
P_{\mathbf{m}}^{\mathrm{ip}}\left(\mathbf{m}+\frac{d}{2}\delta ;\frac{d}{2}\right)
   &=
   \frac{\left(\frac{d}{2}+1\right)_{m_{1}}m_{2}!(m_{1}-m_{2})!}{\left(\frac{d}{2}+1\right)_{m_{1}-m_{2}}}, \nonumber \\
\Phi _{\mathbf{m}}^{(d)}(\mathbf{z})
   &=
   \frac{\left(\frac{d}{2}\right)_{m_{1}-m_{2}}}{(d)_{m_{1}-m_{2}}}
   z_{1}^{m_{1}}z_{2}^{m_{2}}
   {_{2}{F}_1}\left( \begin{matrix} -m_{1}+m_{2},\frac{d}{2} \\ 1-m_{1}+m_{2}-\frac{d}{2} \end{matrix};\frac{z_{2}}{z_{1}}\right), \nonumber \\
\Psi _{\mathbf{m}}^{(d)}(\mathbf{z})
   &=
   \frac{\left(\frac{d}{2}+1\right)_{m_{1}-m_{2}}}{\left(\frac{d}{2}+1\right)_{m_{1}}(m_{1}-m_{2})!m_{2}!}
   z_{1}^{m_{1}}z_{2}^{m_{2}}
   {_{2}{F}_1}\left( \begin{matrix} -m_{1}+m_{2},\frac{d}{2} \\ 1-m_{1}+m_{2}-\frac{d}{2} \end{matrix};\frac{z_{2}}{z_{1}}\right), \nonumber \\
\binom{\mathbf{z}}{\mathbf{k}}^{(d)}
   &=
   \frac{\left(\frac{d}{2}+1\right)_{k_{1}-k_{2}}}{\left(\frac{d}{2}+1\right)_{k_{1}}(k_{1}-k_{2})!k_{2}!}
   (-1)^{k_{1}+k_{2}}\left(-z_{1}-\frac{d}{2}\right)_{k_{2}}(-z_{2})_{k_{1}} \nonumber \\
   & \quad \cdot 
   {_{3}{F}_2}\left( \begin{matrix} -k_{1}+k_{2},\frac{d}{2}, -k_{1}+1+z_{1}\\ 1-k_{1}+m_{2}-\frac{d}{2},-k_{1}+1+z_{2} \end{matrix};1\right), \nonumber \\
{_{0}\mathcal{F}_0}^{(d)}\left( \begin{matrix} \\ \end{matrix};\mathbf{z},\mathbf{u}\right)
   &=
   e^{z_{1}u_{1}+z_{2}u_{2}}
   {_{1}{F}_1}\left( \begin{matrix} \frac{d}{2} \\ d \end{matrix};-(z_{1}-z_{2})(u_{1}-u_{2})\right). \nonumber 
\end{align}

\noindent
{\bf{The $d=2$ case}} In this case, $P_{\mathbf{m}}\left(\mathbf{z};1\right)$ and $P_{\mathbf{m}}^{\mathrm{ip}}\left(\mathbf{z} ;1\right)$ are Schur polynomials and shifted Schur polynomials respectively \cite{OO}. 
\begin{align}
P_{\mathbf{m}}\left(\mathbf{z};1\right)
   &=
   s_{\mathbf{m}}(\mathbf{z})
   =
   \frac{\det\left(z_{i}^{m_{j}+r-j}\right)_{1\leq i,j\leq r}}{\Delta (\mathbf{z})}, \nonumber \\
P_{\mathbf{m}}^{\mathrm{ip}}\left(\mathbf{z} ;1\right)
   &=
   \frac{\det\left(P_{m_{j}+r-j}^{\mathrm{ip}}\left(z_{i}+r-i ;1\right)\right)_{1\leq i,j\leq r}}{\Delta (\mathbf{z})} \nonumber    
\end{align}
where $\Delta (\mathbf{z}):=\prod_{1\leq i<j\leq r}(u_{i}-u_{j})$. 
Further, 
\begin{align}
P_{\mathbf{m}}\left(\mathbf{1};1\right)
   &=
   s_{\mathbf{m}}(\mathbf{1})
   =
   \prod_{1\leq i<j\leq r}
         \frac{\left(j-i\right)_{m_{i}-m_{j}}}{\left(j-i+1\right)_{m_{i}-m_{j}}}, \nonumber \\
P_{\mathbf{m}}^{\mathrm{ip}}\left(\mathbf{m}+\frac{d}{2}\delta ;1\right)
   &=
   \prod_{j=1}^{r}
   \left(r-j+1\right)_{m_{j}}
   \prod_{1\leq i<j\leq r}
         \frac{(j-i )_{m_{i}-m_{j}}}{(j-i+1)_{m_{i}-m_{j}}}, \nonumber \\   
\Phi _{\mathbf{m}}^{(2)}(\mathbf{z})
   &=
   \prod_{1\leq i<j\leq r}
         \frac{\left(j-i+1\right)_{m_{i}-m_{j}}}{\left(j-i\right)_{m_{i}-m_{j}}}
   s_{\mathbf{m}}(\mathbf{z}), \nonumber \\
\Psi _{\mathbf{m}}^{(2)}(\mathbf{z})
   &=
   \prod_{j=1}^{r}
   \frac{1}{\left(r-j+1\right)_{m_{j}}}
   \prod_{1\leq i<j\leq r}
         \frac{\left(j-i+1\right)_{m_{i}-m_{j}}}{\left(j-i\right)_{m_{i}-m_{j}}}
   s_{\mathbf{m}}\left(\mathbf{z}\right), \nonumber \\
\binom{\mathbf{z}}{\mathbf{k}}^{(2)}
   &=
   \frac{1}{\Delta (\mathbf{z})}
   \det\left(\binom{z_{i}+r-i}{k_{j}+r-j}\right)_{1\leq i,j\leq r}, \quad 
{_{0}\mathcal{F}_0}^{(d)}\left( \begin{matrix} \\ \end{matrix};\mathbf{z},\mathbf{u}\right)
   =
   \frac{\det\left(e^{z_{i}u_{j}}\right)_{1\leq i,j\leq r}}{\Delta (\mathbf{z})\Delta (\mathbf{u})}. \nonumber    
\end{align}

\begin{rem}
We remark normalization of various Jack polynomials. 
First, we list some notations of Jack polynomials and their special values at $\mathbf{z}=\mathbf{1}$ (see Table\,1). 
In this article, our notations are based on \cite{FK}. 
In particular, 
\begin{align}
\Psi _{\mathbf{m}}^{(d)}(\mathbf{z})
   =
   d_{\mathbf{m}}
   \frac{1}{\left(\frac{n}{r}\right)_{\mathbf{m}}}
   \Phi _{\mathbf{m}}^{(d)}(\mathbf{z}), \nonumber 
\end{align}
where 
\begin{align}
n
   :=&
   r+\frac{d}{2}r(r-1), \nonumber \\
(\alpha )_{m}
   :=&
   \begin{cases}
   \alpha (\alpha +1) \cdots (\alpha +m-1) & (m \in \mathbb{Z}_{>0})\\
   1 & (m=0)
   \end{cases}, \nonumber \\
(\alpha )_{\mathbf{m}}
   :=&
   \prod_{j=1}^{r}\left(\alpha -\frac{d}{2}(j-1)\right)_{m_{j}}, \nonumber \\
d_{\mathbf{m}}
   :=&
      \prod_{1\leq i<j\leq r}
         \frac{m_{i}-m_{j}+\frac{d}{2}(j-i)}{\frac{d}{2}(j-i)}
         \frac{\left(\frac{d}{2}(j-i+1)\right)_{m_{i}-m_{j}}}{\left(\frac{d}{2}(j-i-1)+1\right)_{m_{i}-m_{j}}} \quad \text{(\cite{FK}, p315)}. \nonumber
\end{align}
From special values of Jack polynomials $P_{\mathbf{m}}\left(\mathbf{z};\frac{d}{2}\right)$ and interpolation Jack polynomials $P_{\mathbf{k}}^{\mathrm{ip}}\left(\mathbf{z}+\frac{d}{2}\delta ;\frac{d}{2}\right)$ (\ref{eq:special value1}) and (\ref{eq:special value2}), we have 
\begin{align}
d_{\mathbf{m}}
   \frac{1}{\left(\frac{n}{r}\right)_{\mathbf{m}}}
   =
   \frac{P_{\mathbf{m}}\left(\mathbf{1};\frac{d}{2}\right)}{P_{\mathbf{m}}^{\mathrm{ip}}\left(\mathbf{m}+\frac{d}{2}\delta ;\frac{d}{2}\right)}. \nonumber 
\end{align}
Next, we remark the relationship between Stanley style $J _{\mathbf{m}}^{\left(\frac{2}{d}\right)}(\mathbf{z})$ and Macdonald style $P_{\mathbf{m}}\left(\mathbf{z};\frac{d}{2}\right)$ 
$$
J _{\mathbf{m}}^{\left(\frac{2}{d}\right)}(\mathbf{z})
   =
   \left(\frac{2}{d}\right)^{|\mathbf{m}|}
   \prod_{(i,j) \in \mathbf{m}}
      \left(m_{i}-j+\frac{d}{2}(m_{j}^{\prime}-i+1)\right)
   P_{\mathbf{m}}\left(\mathbf{z};\frac{d}{2}\right) \quad \text{(\cite{M}\,V\!I\,(10.22))}
$$
where a partition $\mathbf{m}$ is identified with its diagram : 
$$
\mathbf{m}
   =\{s=(i,j)\mid 1\leq i\leq r,1\leq j\leq m_{i}\}.
$$
Hence we have
\begin{align}
\Psi _{\mathbf{m}}^{(d)}(\mathbf{z})
   &=
   \frac{P_{\mathbf{m}}\left(\mathbf{z};\frac{d}{2}\right)}{P_{\mathbf{m}}^{\mathrm{ip}}\left(\mathbf{m}+\frac{d}{2}\delta ;\frac{d}{2}\right)} \nonumber \\
   &=
   \left(\frac{d}{2}\right)^{|\mathbf{m}|}
   \prod_{(i,j) \in \mathbf{m}}
      \frac{1}{\left(m_{i}-j+\frac{d}{2}(m_{j}^{\prime}-i+1)\right)}    
   \frac{1}{P_{\mathbf{m}}^{\mathrm{ip}}\left(\mathbf{m}+\frac{d}{2}\delta ;\frac{d}{2}\right)}
   J _{\mathbf{m}}^{\left(\frac{2}{d}\right)}(\mathbf{z}). \nonumber
\end{align}
The relationship between $C _{\mathbf{m}}^{\left(\frac{2}{d}\right)}(\mathbf{1})$ (Kaneko style) and $\Psi _{\mathbf{m}}^{(d)}(\mathbf{1})$ (our style)
\begin{align}
C _{\mathbf{m}}^{\left(\frac{2}{d}\right)}(\mathbf{1})
   &=
   |\mathbf{m}|!\prod_{(i,j) \in \mathbf{m}}
      \frac{\left(j-1+\frac{d}{2}(r-i+1)\right)}{\left(m_{i}-j+\frac{d}{2}(m_{j}^{\prime}-i+1)\right)\left(m_{i}-j+1+\frac{d}{2}(m_{j}^{\prime}-i)\right)} \nonumber \\
   &=
      |\mathbf{m}|!\frac{P_{\mathbf{m}}\left(\mathbf{1};\frac{d}{2}\right)}{P_{\mathbf{m}}^{\mathrm{ip}}\left(\mathbf{m}+\frac{d}{2}\delta ;\frac{d}{2}\right)}. \nonumber
\end{align}
follows from (\ref{eq:special value1}), (\ref{eq:special value2}) and \cite{Ka} (18). 
Thus, we have 
\begin{align}
\Psi _{\mathbf{m}}^{(d)}(\mathbf{z})
   =
   \frac{1}{|\mathbf{m}|!}C_{\mathbf{m}}^{\left(\frac{2}{d}\right)}(\mathbf{z}). \nonumber 
\end{align}
To summarize the above results, we obtain 
\begin{align}
\Psi _{\mathbf{m}}^{(d)}(\mathbf{z})
   =
   d_{\mathbf{m}}
   \frac{1}{\left(\frac{n}{r}\right)_{\mathbf{m}}}
   \Phi _{\mathbf{m}}^{(d)}(\mathbf{z})
   =
   \frac{1}{P_{\mathbf{m}}^{\mathrm{ip}}\left(\mathbf{m}+\frac{d}{2}\delta ;\frac{d}{2}\right)}
   P_{\mathbf{m}}\left(\mathbf{z};\frac{d}{2}\right)
   =
   \frac{1}{|\mathbf{m}|!}C_{\mathbf{m}}^{\left(\frac{2}{d}\right)}(\mathbf{z}). 
\end{align}
\end{rem}

\begin{table}[htb]
\begin{center}
  \begin{tabular}{|c||c|c|} \hline
  & notation & special value at $\mathbf{z}=\mathbf{1}$ \\ \hline \hline 
   Faraut-Kor\'{a}nyi & $\Phi _{\mathbf{m}}^{(d)}(\mathbf{z})$ & $1$ \\ \hline
   Stanley & $J _{\mathbf{m}}^{\left(\frac{2}{d}\right)}(\mathbf{z})$ & $\left(\frac{2}{d}\right)^{|\mathbf{m}|}\prod_{(i,j) \in \mathbf{m}}\left(j-1+\frac{d}{2}(r-i+1)\right)$ (\cite{S}\,Thm.\,5.4) \\ \hline 
   Macdonald & $P_{\mathbf{m}}\left(\mathbf{z};\frac{d}{2}\right)$ & $\prod_{(i,j) \in \mathbf{m}}
      \frac{j-1+\frac{d}{2}(r-i+1)}{m_{i}-j+\frac{d}{2}(m_{j}^{\prime}-i+1)}$ (\cite{M}\,V\!I\,(10.20))\\ \hline 
   Kaneko & $C _{\mathbf{m}}^{\left(\frac{2}{d}\right)}(\mathbf{z})$ & $|\mathbf{m}|!\prod_{(i,j) \in \mathbf{m}}
      \frac{\left(j-1+\frac{d}{2}(r-i+1)\right)}{\left(m_{i}-j+\frac{d}{2}(m_{j}^{\prime}-i+1)\right)\left(m_{i}-j+1+\frac{d}{2}(m_{j}^{\prime}-i)\right)}$ (\cite{Ka}\,(18))\\ \hline 
   S & $\Psi _{\mathbf{m}}^{(d)}(\mathbf{z})$ & $\prod_{(i,j) \in \mathbf{m}}
      \frac{\left(j-1+\frac{d}{2}(r-i+1)\right)}{\left(m_{i}-j+\frac{d}{2}(m_{j}^{\prime}-i+1)\right)\left(m_{i}-j+1+\frac{d}{2}(m_{j}^{\prime}-i)\right)}$ \\ \hline
  \end{tabular}
\end{center}
\caption{Notations and normalizations of Jack polynomials}
\end{table}

Under the following, we provide all necessary formulas to prove our main results.\\
\noindent
\underline{{\bf{Pieri type formulas for Jack polynomials}}} 
For any partition $\mathbf{m} \in \mathcal{P}$, 
\begin{align}
\label{eq:key formula 4}
e^{|\mathbf{u}|}\Psi _{\mathbf{m}}^{(d)}(\mathbf{u})
   &=
   \sum_{\mathbf{n} \in \mathcal{P}}
   \binom{\mathbf{n}}{\mathbf{m}}^{(d)}\Psi _{\mathbf{n}}^{(d)}(\mathbf{u}) \quad (\text{\cite{L} Section\,14}).
\end{align}
Since
$$
e^{|\mathbf{u}|}=\sum_{N\geq 0}\frac{1}{N!}|\mathbf{u}|^{N}
$$
and comparing the terms of degree $N+|\mathbf{m}|$ in (\ref{eq:key formula 4}), we have 
\begin{align}
\label{eq:key formula 5}
\frac{|\mathbf{u}|^{N}}{N!}\Psi _{\mathbf{m}}^{(d)}(\mathbf{u})
   &=
   \sum_{\substack{|\mathbf{n}|-|\mathbf{m}|=N, \\ \mathbf{n} \in \mathcal{P}}}
   \binom{\mathbf{n}}{\mathbf{m}}^{(d)}\Psi _{\mathbf{n}}^{(d)}(\mathbf{u}).
\end{align}
From \cite{L} Section\,14, 
$$
\binom{\mathbf{m}^{i}}{\mathbf{m}}^{(d)}
   =
      \left(m_{i}+1+\frac{d}{2}(r-i)\right)
      h_{-,i}^{(d)}(\mathbf{m}^{i}), 
$$
where
$\epsilon_{i}:=(0,\ldots,0,\stackrel{i}{\stackrel{\vee}{1}},0,\ldots,0) \in \mathbb{Z}^{r}$, 
$\mathbf{m}^{i}:=\mathbf{m}+\epsilon_{i}$ and  
\begin{align}
h_{\pm ,i}^{(d)}(\mathbf{m})
   &:=
      \prod_{1\leq k\not=i\leq r}\frac{m_{i}-m_{k}-\frac{d}{2}(i-k)\pm \frac{d}{2}}{m_{i}-m_{k}-\frac{d}{2}(i-k)}. \nonumber 
\end{align}
% We remark that if $\mathbf{m}^{i} \not\in \mathcal{P}$ then $h_{-,i}^{(d)}(\mathbf{m}^{i})=0$. 
In particular, the $N=1$ case of (\ref{eq:key formula 5}) is the following :  
\begin{align}
\label{eq:Pieri 2}
|\mathbf{u}|\Psi_{\mathbf{m}}^{(d)}(\mathbf{u})
   &=
      \sum_{\substack{1\leq i\leq r, \\ \mathbf{m}^{i} \in \mathcal{P}}}
      \binom{\mathbf{m}^{i}}{\mathbf{m}}^{(d)}
         \Psi_{\mathbf{m}^{i}}^{(d)}(\mathbf{u})
   =   
      \sum_{\substack{1\leq i\leq r, \\ \mathbf{m}^{i} \in \mathcal{P}}}
         \Psi_{\mathbf{m}^{i}}^{(d)}(\mathbf{u})
         \left(m_{i}+1+\frac{d}{2}(r-i)\right)
         h_{-,i}^{(d)}(\mathbf{m}^{i}).  
\end{align}
\underline{{\bf{Properties of ${_{0}\mathcal{F}_0}^{(d)}$}}} By \cite{L} Section\,14, we have 
\begin{align}
\label{eq:key formula 3}
E_{0}(\mathbf{z}){_{0}\mathcal{F}_0}^{(d)}\left( \begin{matrix} \\ \end{matrix};\mathbf{z},\mathbf{u}\right)
   &=
   {_{0}\mathcal{F}_0}^{(d)}\left( \begin{matrix} \\ \end{matrix};\mathbf{z},\mathbf{u}\right)|\mathbf{u}|. 
\end{align}
In particular, we obtain the index law of ${_{0}\mathcal{F}_0}^{(d)}\left( \begin{matrix} \\ \end{matrix};\mathbf{z},\mathbf{u}\right)$
\begin{align}
\label{eq:key formula 2}
{_{0}\mathcal{F}_0}^{(d)}\left( \begin{matrix} \\ \end{matrix};\mathbf{1}+\mathbf{z},\mathbf{u}\right)
   &=
   e^{E_{0}(\mathbf{z})}
   {_{0}\mathcal{F}_0}^{(d)}\left( \begin{matrix} \\ \end{matrix};\mathbf{z},\mathbf{u}\right)
   =
   {_{0}\mathcal{F}_0}^{(d)}\left( \begin{matrix} \\ \end{matrix};\mathbf{z},\mathbf{u}\right)
   e^{|\mathbf{u}|}.
\end{align}
\noindent
\underline{{\bf{Other formula}}}
\begin{align}
\label{eq:key formula 1}
|\mathbf{u}|^{N}
   &=
   N!\sum_{\substack{|\mathbf{m}|=N, \\ \mathbf{m} \in \mathcal{P}}}
   \Psi _{\mathbf{m}}^{(d)}(\mathbf{u}) \quad (\text{\cite{S}\,Prop.\,2.3 or \cite{Ka}\,(17)}).
\end{align}
To summarize the above results, we obtain the following dictionary. 
\begin{align}
P_{m}\left(1;\frac{d}{2}\right)=1 
   \quad &\Rightarrow \quad 
   P_{\mathbf{m}}\left(\mathbf{1};\frac{d}{2}\right)
   =
   \prod_{1\leq i<j\leq r}
         \frac{\left(\frac{d}{2}(j-i+1)\right)_{m_{i}-m_{j}}}{\left(\frac{d}{2}(j-i)\right)_{m_{i}-m_{j}}}, \nonumber \\
P_{m}^{\mathrm{ip}}\left(m ;\frac{d}{2}\right)=m!         
   \quad &\Rightarrow \quad 
P_{\mathbf{m}}^{\mathrm{ip}}\left(\mathbf{m}+\frac{d}{2}\delta ;\frac{d}{2}\right)
   =
   \prod_{j=1}^{r}
   \left(\frac{d}{2}(r-j)+1\right)_{m_{j}}
   \nonumber \\
   & \quad \quad \quad \quad \cdot 
   \prod_{1\leq i<j\leq r}\!\!\!\!\!
         \frac{\left(\frac{d}{2}(j-i-1)+1\right)_{m_{i}-m_{j}}}{\left(\frac{d}{2}(j-i)+1\right)_{m_{i}-m_{j}}}, \nonumber \\
\Phi _{m}^{(d)}(z):=z^{m} \quad &\Rightarrow \quad \Phi _{\mathbf{m}}^{(d)}(\mathbf{z}):=\frac{P_{\mathbf{m}}\left(\mathbf{z} ;\frac{d}{2}\right)}{P_{\mathbf{m}}\left(\mathbf{1} ;\frac{d}{2}\right)}, \nonumber \\
\Psi _{m}^{(d)}(z):=\frac{z^{m}}{m!} \quad &\Rightarrow \quad \Psi _{\mathbf{m}}^{(d)}(\mathbf{z})
   :=
   \frac{P_{\mathbf{m}}\left(\mathbf{1};\frac{d}{2}\right)\Phi _{\mathbf{m}}^{(d)}(\mathbf{z})}{P_{\mathbf{m}}^{\mathrm{ip}}\left(\mathbf{m}+\frac{d}{2}\delta ;\frac{d}{2}\right)}
   =
   \frac{P_{\mathbf{m}}\left(\mathbf{z};\frac{d}{2}\right)}{P_{\mathbf{m}}^{\mathrm{ip}}\left(\mathbf{m}+\frac{d}{2}\delta ;\frac{d}{2}\right)}, \nonumber \\
\binom{m}{k}:=\frac{P_{k}^{\mathrm{ip}}\left(m ;\frac{d}{2}\right)}{P_{k}^{\mathrm{ip}}\left(k ;\frac{d}{2}\right)}
 \quad &\Rightarrow \quad 
   \binom{\mathbf{m}}{\mathbf{k}}^{(d)}
   :=
   \frac{P_{\mathbf{k}}^{\mathrm{ip}}\left(\mathbf{m}+\frac{d}{2}\delta ;\frac{d}{2}\right)}{P_{\mathbf{k}}^{\mathrm{ip}}\left(\mathbf{k}+\frac{d}{2}\delta ;\frac{d}{2}\right)}, \nonumber \\
e^{zu}
   =
   \sum_{m=0}^{\infty}
   \frac{1}{m!}z^{m}u^{m}
 \quad &\Rightarrow \quad 
{_{0}\mathcal{F}_0}^{(d)}\left( \begin{matrix} \\ \end{matrix};\mathbf{z},\mathbf{u}\right)
   :=
   \sum_{\mathbf{m} \in \mathcal{P}}
   \Psi_{\mathbf{k}}^{(d)}(\mathbf{z})\Phi_{\mathbf{k}}^{(d)}(\mathbf{u}) \nonumber \\
e^{u}\frac{u^{m}}{m!}=\sum_{n=0}^{\infty}\binom{n}{m}\frac{u^{n}}{n!}
   \quad &\Rightarrow \quad 
e^{|\mathbf{u}|}\Psi _{\mathbf{m}}^{(d)}(\mathbf{u})
   =
   \sum_{\mathbf{n} \in \mathcal{P}}
   \binom{\mathbf{n}}{\mathbf{m}}^{(d)}\Psi _{\mathbf{n}}^{(d)}(\mathbf{u}), \nonumber \\
\frac{u^{N}}{N!}\frac{u^{m}}{m!}=\binom{N+m}{m}\frac{u^{N+m}}{(N+m)!}
   \quad &\Rightarrow \quad 
\frac{|\mathbf{u}|^{N}}{N!}\Psi _{\mathbf{m}}^{(d)}(\mathbf{u})
   =
   \sum_{\substack{|\mathbf{n}|-|\mathbf{m}|=N, \\ \mathbf{n} \in \mathcal{P}}}
   \binom{\mathbf{n}}{\mathbf{m}}^{(d)}\Psi _{\mathbf{n}}^{(d)}(\mathbf{u}), \nonumber \\
u\frac{u^{m}}{m!}=\frac{u^{m+1}}{(m+1)!}(m+1)
   \quad &\Rightarrow \quad 
|\mathbf{u}|\Psi_{\mathbf{m}}^{(d)}(\mathbf{u}) 
   =
      \sum_{\substack{1\leq i\leq r, \\ \mathbf{m}^{i} \in \mathcal{P}}}
         \Psi_{\mathbf{m}^{i}}^{(d)}(\mathbf{u})
         \left(m_{i}+1+\frac{d}{2}(r-i)\right)
         h_{-,i}^{(d)}(\mathbf{m}^{i}), \nonumber \\
\partial_{z}e^{zu}=e^{zu}u
   \quad &\Rightarrow \quad 
E_{0}(\mathbf{z}){_{0}\mathcal{F}_0}^{(d)}\left( \begin{matrix} \\ \end{matrix};\mathbf{z},\mathbf{u}\right)
   =
   {_{0}\mathcal{F}_0}^{(d)}\left( \begin{matrix} \\ \end{matrix};\mathbf{z},\mathbf{u}\right)|\mathbf{u}|, \nonumber \\
e^{(1+z)u}=e^{u}e^{zu}
   \quad &\Rightarrow \quad 
{_{0}\mathcal{F}_0}^{(d)}\left( \begin{matrix} \\ \end{matrix};\mathbf{1}+\mathbf{z},\mathbf{u}\right)
   =
   e^{|\mathbf{u}|}
   {_{0}\mathcal{F}_0}^{(d)}\left( \begin{matrix} \\ \end{matrix};\mathbf{z},\mathbf{u}\right), \nonumber \\
u^{N}=N!\frac{u^{N}}{N!}
   \quad &\Rightarrow \quad 
|\mathbf{u}|^{N}
   =
   N!\sum_{\substack{|\mathbf{m}|=N, \\ \mathbf{m} \in \mathcal{P}}}
   \Psi _{\mathbf{m}}^{(d)}(\mathbf{u}). \nonumber
\end{align}
% Although these multivariate analogue is very complicated, the following two special cases are written down explicitly. \\

\section{Multivariate Bernoulli polynomials}
We define multivariate Bernoulli polynomials $B_{\mathbf{m}}^{(d)}(\mathbf{z})$ or $B_{\mathbf{m}}\left(\mathbf{z};\frac{d}{2}\right)$ by the following generating function. 
% \begin{align}
% \label{eq:def of multivariate Bernoulli}
% \frac{u}{e^{u}-1}e^{zu}
%    =
%    \sum_{m=0}^{\infty}
%    B_{m}(z)\Psi _{m}(u)
%    \quad \Rightarrow \quad 
% \frac{|\mathbf{u}|}{e^{|\mathbf{u}|}-1}
%    {_{0}\mathcal{F}_0}^{(d)}\left( \begin{matrix} \\ \end{matrix};\mathbf{z},\mathbf{u}\right)
%    =
%    \sum_{\mathbf{m} \in \mathcal{P}}
%       B_{\mathbf{m}}^{(d)}\left(\mathbf{z}\right)
%       \Psi _{\mathbf{m}}^{(d)}(\mathbf{u}).
% \end{align}
\begin{align}
\frac{u}{e^{u}-1}e^{zu}
   &=
   \sum_{m=0}^{\infty}
   B_{m}(z)\Psi _{m}(u) \quad (|u|<2\pi )\nonumber \\
   & \quad \quad \Downarrow \quad \nonumber \\
\label{eq:def of multivariate Bernoulli}
\frac{|\mathbf{u}|}{e^{|\mathbf{u}|}-1}
   {_{0}\mathcal{F}_0}^{(d)}\left( \begin{matrix} \\ \end{matrix};\mathbf{z},\mathbf{u}\right)
   &=
   \sum_{\mathbf{m} \in \mathcal{P}}
      B_{\mathbf{m}}^{(d)}\left(\mathbf{z}\right)
      \Psi _{\mathbf{m}}^{(d)}(\mathbf{u}) \quad (|u_{1}+\cdots +u_{r}|<2\pi ).
\end{align}
\begin{rem}
Originally, we consider the following type generating function and multivariate analogue of Bernoulli polynomials. 
$$
\prod_{j=1}^{r}\frac{u_{j}}{e^{u_{j}}-1}
   {_{0}\mathcal{F}_0}^{(d)}\left( \begin{matrix} \\ \end{matrix};\mathbf{z},\mathbf{u}\right)
   =
   \sum_{\mathbf{m} \in \mathcal{P}}
      \widetilde{B}_{\mathbf{m}}^{(d)}\left(\mathbf{z}\right)
      \Psi _{\mathbf{m}}^{(d)}(\mathbf{u})
$$
In the $d=2$ case, this type generating function has the determinant expression 
$$
\prod_{j=1}^{r}\frac{u_{j}}{e^{u_{j}}-1}
   {_{0}\mathcal{F}_0}^{(2)}\left( \begin{matrix} \\ \end{matrix};\mathbf{z},\mathbf{u}\right)
   =
   \frac{\det\left(\frac{u_{j}e^{z_{i}u_{j}}}{e^{u_{j}}-1}\right)_{1\leq i,j\leq r}}{\Delta (\mathbf{z})\Delta (\mathbf{u})}
$$
and $\widetilde{B}_{\mathbf{m}}^{(2)}\left(\mathbf{z}\right)$ has the Jacobi-Trudi type formula
$$
\widetilde{B}_{\mathbf{m}}^{(2)}\left(\mathbf{z}\right)
   =
   \frac{\det\left(B_{m_{i}+r-i}(z_{j})\right)}{\Delta (\mathbf{z})}.
$$
However, for this multivariate analogue of Bernoulli polynomials, we can not find an analogue of the formulas (\ref{eq:found prop1}) - (\ref{eq:found prop7}). 
Therefore, we investigate the above type (\ref{eq:def of multivariate Bernoulli}) multivariate Bernoulli polynomials. 
\end{rem}
\begin{thm}
{\rm{(1)}} Special value at $\mathbf{z}=0$
\begin{align}
B_{\mathbf{m}}^{(d)}\left(\mathbf{0}\right)
   =
   B_{|\mathbf{m}|}.
\end{align}
{\rm{(2)}} Difference equation 
\begin{align}
B_{\mathbf{m}}^{(d)}\left(\mathbf{z}+\mathbf{1}\right)-B_{\mathbf{m}}^{(d)}\left(\mathbf{z}\right)
   =
   \sum_{i=1}^{r}
      \Phi_{\mathbf{m}_{i}}^{(d)}(\mathbf{z})
      \left(m_{i}+\frac{d}{2}(r-i)\right)
      h_{-,i}^{(d)}(\mathbf{m}).
\end{align}
{\rm{(3)}} Differential equation 
\begin{align}
E_{0}(\mathbf{z})B_{\mathbf{m}}^{(d)}\left(\mathbf{z}\right)
   =
      \sum_{i=1}^{r}
      B_{\mathbf{m}_{i}}^{(d)}\left(\mathbf{z}\right)
         \left(m_{i}+\frac{d}{2}(r-i)\right)
         h_{-,i}^{(d)}(\mathbf{m}).
\end{align}
{\rm{(4)}} Symmetry
\begin{align}
B_{\mathbf{m}}^{(d)}(\mathbf{1}-\mathbf{z})
   =
   (-1)^{|\mathbf{m}|}B_{\mathbf{m}}^{(d)}(\mathbf{z})
\end{align}
{\rm{(5)}} Explicit formula
% \begin{align}
% B_{\mathbf{m}}^{(d)}(\mathbf{z})
%    =
%    \sum_{N=0}^{|\mathbf{m}|}
%    B_{N}
%    \sum_{\mathbf{n}\subset \mathbf{m}, |\mathbf{m}|-|\mathbf{n}|=N}
%    \binom{\mathbf{m}}{\mathbf{n}}^{(d)}
%    \Phi _{\mathbf{n}}^{(d)}(\mathbf{z}).
% \end{align}
\begin{align}
B_{\mathbf{m}}^{(d)}(\mathbf{z})
   =
   \sum_{\mathbf{n}\subset \mathbf{m}}
   B_{|\mathbf{m}|-|\mathbf{n}|}
   \binom{\mathbf{m}}{\mathbf{n}}^{(d)}
   \Phi _{\mathbf{n}}^{(d)}(\mathbf{z}).
\end{align}
{\rm{(6)}} Inversion formula 
\begin{align}
\Phi _{\mathbf{m}}^{(d)}(\mathbf{z})
   =
   \sum_{\mathbf{n}\subset \mathbf{m}}
   \frac{1}{|\mathbf{m}|-|\mathbf{n}|+1}
   \binom{\mathbf{m}}{\mathbf{n}}^{(d)}
   B_{\mathbf{n}}^{(d)}(\mathbf{z}). 
\end{align}
{\rm{(7)}} Multiplication formula
\begin{align}
\sum_{i=0}^{N-1}
   B_{\mathbf{m}}^{(d)}\left(\mathbf{z}+\frac{i}{N}\mathbf{1}\right)
   &=
   N^{1-|\mathbf{m}|}B_{\mathbf{m}}^{(d)}(N\mathbf{z})
\end{align}
{\rm{(8)}} Binomial formula
\begin{align}
B_{\mathbf{m}}^{(d)}(\mathbf{z}+\mathbf{1})
   &=
\sum_{\mathbf{n} \subset \mathbf{m}}
   \binom{\mathbf{m}}{\mathbf{n}}^{(d)}
   B_{\mathbf{n}}^{(d)}(\mathbf{z})
\end{align}
\end{thm}
\begin{proof}
{\rm{(1)}} By the definition of the multivariate Bernoulli polynomials and Bernoulli numbers, we have
\begin{align}
\sum_{\mathbf{m} \in \mathcal{P}}
   B_{\mathbf{m}}^{(d)}\left(\mathbf{0}\right)
   \Psi _{\mathbf{m}}^{(d)}(\mathbf{u})
   =
   \frac{|\mathbf{u}|}{e^{|\mathbf{u}|}-1}
   =
   \sum_{N=0}^{\infty}
      \frac{B_{N}}{N!}|\mathbf{u}|^{N}. \nonumber
\end{align}
On the other hand, by (\ref{eq:key formula 1})
\begin{align}
% \sum_{\mathbf{m} \in \mathcal{P}}
%    B_{\mathbf{m}}^{(d)}\left(\mathbf{0}\right)
%    \Psi _{\mathbf{m}}^{(d)}(\mathbf{u})
   \sum_{N=0}^{\infty}
      \frac{B_{N}}{N!}|\mathbf{u}|^{N}
   =
\sum_{N=0}^{\infty}
      B_{N}
   \sum_{|\mathbf{m}|=N}
   \Psi _{\mathbf{m}}^{(d)}(\mathbf{u})
   =
\sum_{\mathbf{m} \in \mathcal{P}}
   B_{|\mathbf{m}|}
   \Psi _{\mathbf{m}}^{(d)}(\mathbf{u}). \nonumber   
\end{align}
{\rm{(2)}} By (\ref{eq:key formula 2}) and (\ref{eq:Pieri 2}), we have
\begin{align}
\sum_{\mathbf{m} \in \mathcal{P}}
      \left(B_{\mathbf{m}}^{(d)}\left(\mathbf{z}+\mathbf{1}\right)-B_{\mathbf{m}}^{(d)}\left(\mathbf{z}\right)\right)
      \Psi _{\mathbf{m}}^{(d)}(\mathbf{z})
   &=
   \frac{|\mathbf{u}|}{e^{|\mathbf{u}|}-1}
   \left({_{0}\mathcal{F}_0}^{(d)}\left( \begin{matrix} \\ \end{matrix};\mathbf{z}+\mathbf{1},\mathbf{u}\right)-{_{0}\mathcal{F}_0}^{(d)}\left( \begin{matrix} \\ \end{matrix};\mathbf{z},\mathbf{u}\right)\right) \nonumber \\
   &=
   \frac{|\mathbf{u}|}{e^{|\mathbf{u}|}-1}
   \left(e^{|\mathbf{u}|}{_{0}\mathcal{F}_0}^{(d)}\left( \begin{matrix} \\ \end{matrix};\mathbf{z},\mathbf{u}\right)-{_{0}\mathcal{F}_0}^{(d)}\left( \begin{matrix} \\ \end{matrix};\mathbf{z},\mathbf{u}\right)\right) \nonumber \\
   &=
   |\mathbf{u}|{_{0}\mathcal{F}_0}^{(d)}\left( \begin{matrix} \\ \end{matrix};\mathbf{z},\mathbf{u}\right) \nonumber \\
   &=
      \sum_{\mathbf{m} \in \mathcal{P}}
   \Phi_{\mathbf{m}}^{(d)}(\mathbf{z})|\mathbf{u}|\Psi_{\mathbf{m}}^{(d)}(\mathbf{u}) \nonumber \\
   &=
      \sum_{\mathbf{m} \in \mathcal{P}}
         \Phi_{\mathbf{m}}^{(d)}(\mathbf{z})
         \sum_{i=1}^{r}
         \Psi_{\mathbf{m}^{i}}^{(d)}(\mathbf{u})
         \left(m_{i}+1+\frac{d}{2}(r-i)\right)
         h_{-,i}^{(d)}(\mathbf{m}^{i}) \nonumber \\
   &=
      \sum_{\mathbf{m} \in \mathcal{P}}
      \sum_{i=1}^{r}
         \Phi_{\mathbf{m}_{i}}^{(d)}(\mathbf{z})
         \left(m_{i}+\frac{d}{2}(r-i)\right)
         h_{-,i}^{(d)}(\mathbf{m})
         \Psi_{\mathbf{m}}^{(d)}(\mathbf{u}). \nonumber         
\end{align}
{\rm{(3)}} By (\ref{eq:key formula 3}) and (\ref{eq:Pieri 2}), 
\begin{align}
\sum_{\mathbf{m} \in \mathcal{P}}
   E_{0}(\mathbf{z})B_{\mathbf{m}}^{(d)}\left(\mathbf{z}\right)
      \Psi _{\mathbf{m}}^{(d)}(\mathbf{u})
   &=
\frac{|\mathbf{u}|}{e^{|\mathbf{u}|}-1}
E_{0}(\mathbf{z})
   {_{0}\mathcal{F}_0}^{(d)}\left( \begin{matrix} \\ \end{matrix};\mathbf{z},\mathbf{u}\right) \nonumber \\
   &=
\frac{|\mathbf{u}|}{e^{|\mathbf{u}|}-1}
   {_{0}\mathcal{F}_0}^{(d)}\left( \begin{matrix} \\ \end{matrix};\mathbf{z},\mathbf{u}\right)
   |\mathbf{u}| \nonumber \\
   &=
   \sum_{\mathbf{m} \in \mathcal{P}}
      B_{\mathbf{m}}^{(d)}\left(\mathbf{z}\right)
   |\mathbf{u}|\Psi_{\mathbf{m}}^{(d)}(\mathbf{u}) \nonumber \\
   &=
   \sum_{\mathbf{m} \in \mathcal{P}}
      B_{\mathbf{m}}^{(d)}\left(\mathbf{z}\right)
      \sum_{i=1}^{r}
         \Psi_{\mathbf{m}^{i}}^{(d)}(\mathbf{u})
         \left(m_{i}+1+\frac{d}{2}(r-i)\right)
         h_{-,i}^{(d)}(\mathbf{m}^{i}) \nonumber \\
   &=
   \sum_{\mathbf{m} \in \mathcal{P}}
      \sum_{i=1}^{r}
      B_{\mathbf{m}_{i}}^{(d)}(\mathbf{z})
         \left(m_{i}+\frac{d}{2}(r-i)\right)
         h_{-,i}^{(d)}(\mathbf{m})
         \Psi_{\mathbf{m}}^{(d)}(\mathbf{u}). \nonumber
\end{align}
{\rm{(4)}} By (\ref{eq:key formula 2}), 
\begin{align}
\sum_{\mathbf{m} \in \mathcal{P}}
   B_{\mathbf{m}}^{(d)}\left(\mathbf{1}-\mathbf{z}\right)
      \Psi _{\mathbf{m}}^{(d)}(\mathbf{u})
   &=
\frac{|\mathbf{u}|}{e^{|\mathbf{u}|}-1}
   {_{0}\mathcal{F}_0}^{(d)}\left( \begin{matrix} \\ \end{matrix};\mathbf{1}-\mathbf{z},\mathbf{u}\right) \nonumber \\
   &=
\frac{|\mathbf{u}|}{e^{|\mathbf{u}|}-1}
   {_{0}\mathcal{F}_0}^{(d)}\left( \begin{matrix} \\ \end{matrix};\mathbf{z},-\mathbf{u}\right)
   e^{|\mathbf{u}|} \nonumber \\
   &=
\frac{|-\mathbf{u}|}{e^{|-\mathbf{u}|}-1}
   {_{0}\mathcal{F}_0}^{(d)}\left( \begin{matrix} \\ \end{matrix};\mathbf{z},-\mathbf{u}\right) \nonumber \\
   &=
   \sum_{\mathbf{m} \in \mathcal{P}}
      B_{\mathbf{m}}^{(d)}\left(\mathbf{z}\right)
      \Psi_{\mathbf{m}}^{(d)}(-\mathbf{u}) \nonumber \\
   &=
   \sum_{\mathbf{m} \in \mathcal{P}}
      (-1)^{|\mathbf{m}|}
      B_{\mathbf{m}}^{(d)}\left(\mathbf{z}\right)
      \Psi_{\mathbf{m}}^{(d)}(\mathbf{u}). \nonumber
\end{align}
{\rm{(5)}} By (\ref{eq:key formula 5}),
\begin{align}
\sum_{\mathbf{m} \in \mathcal{P}}
   B_{\mathbf{m}}^{(d)}\left(\mathbf{z}\right)
   \Psi _{\mathbf{m}}^{(d)}(\mathbf{u})
   &=
   \frac{|\mathbf{u}|}{e^{|\mathbf{u}|}-1}
   {_{0}\mathcal{F}_0}^{(d)}\left( \begin{matrix} \\ \end{matrix};\mathbf{z},\mathbf{u}\right) \nonumber \\
   &=
   \sum_{N=0}^{\infty}
      \frac{B_{N}}{N!}|\mathbf{u}|^{N}
      \sum_{\mathbf{n} \in \mathcal{P}}
      \Phi _{\mathbf{n}}^{(d)}(\mathbf{z})\Psi _{\mathbf{n}}^{(d)}(\mathbf{u}) \nonumber \\
   &=
   \sum_{N=0}^{\infty}
      B_{N}
      \sum_{\mathbf{n} \in \mathcal{P}}
      \Phi _{\mathbf{n}}^{(d)}(\mathbf{z})
      \frac{|\mathbf{u}|^{N}}{N!}\Psi _{\mathbf{n}}^{(d)}(\mathbf{u}) \nonumber \\
   &=
   \sum_{N=0}^{\infty}
      B_{N}
      \sum_{\mathbf{n} \in \mathcal{P}}
      \Phi _{\mathbf{n}}^{(d)}(\mathbf{z})
      \sum_{|\mathbf{m}|-|\mathbf{n}|=N}
   \binom{\mathbf{m}}{\mathbf{n}}^{(d)}\Psi _{\mathbf{m}}^{(d)}(\mathbf{u}) \nonumber \\
   &=
   \sum_{\mathbf{m} \in \mathcal{P}}
   \sum_{\mathbf{n}\subset \mathbf{m}}
   B_{|\mathbf{m}|-|\mathbf{n}|}
   \binom{\mathbf{m}}{\mathbf{n}}^{(d)}
   \Phi _{\mathbf{n}}^{(d)}(\mathbf{z})
   \Psi _{\mathbf{m}}^{(d)}(\mathbf{u}). \nonumber
%    &=
%    \sum_{\mathbf{m} \in \mathcal{P}}
%    \sum_{N=0}^{|\mathbf{m}|}
%    B_{N}
%    \sum_{\mathbf{n}\subset \mathbf{m}, |\mathbf{m}|-|\mathbf{n}|=N}
%    \binom{\mathbf{m}}{\mathbf{n}}^{(d)}
%    \Phi _{\mathbf{n}}^{(d)}(\mathbf{z})
%    \Psi _{\mathbf{m}}^{(d)}(\mathbf{u}). \nonumber
\end{align}
{\rm{(6)}} By (\ref{eq:key formula 5}),
\begin{align}
\sum_{\mathbf{m} \in \mathcal{P}}
   \Phi _{\mathbf{m}}^{(d)}\left(\mathbf{z}\right)
   \Psi _{\mathbf{m}}^{(d)}(\mathbf{u})
   &=
   {_{0}\mathcal{F}_0}^{(d)}\left( \begin{matrix} \\ \end{matrix};\mathbf{z},\mathbf{u}\right) \nonumber \\
   &=
   \frac{e^{|\mathbf{u}|}-1}{|\mathbf{u}|}
   \sum_{\mathbf{n} \in \mathcal{P}}
   B_{\mathbf{n}}^{(d)}\left(\mathbf{z}\right)
   \Psi _{\mathbf{n}}^{(d)}(\mathbf{u}) \nonumber \\
   &=
   \sum_{N=0}^{\infty}
   \frac{1}{N+1}\frac{1}{N!}|\mathbf{u}|^{N}
   \sum_{\mathbf{n} \in \mathcal{P}}
   B_{\mathbf{n}}^{(d)}\left(\mathbf{z}\right)
   \Psi _{\mathbf{n}}^{(d)}(\mathbf{u}) \nonumber \\
   &=
   \sum_{N=0}^{\infty}
   \frac{1}{N+1}
   \sum_{\mathbf{n} \in \mathcal{P}}
   B_{\mathbf{n}}^{(d)}\left(\mathbf{z}\right)
   \frac{|\mathbf{u}|^{N}}{N!}\Psi _{\mathbf{n}}^{(d)}(\mathbf{u}) \nonumber \\
   &=
   \sum_{N=0}^{\infty}
   \frac{1}{N+1}
   \sum_{\mathbf{n} \in \mathcal{P}}
   B_{\mathbf{n}}^{(d)}(\mathbf{z})
   \sum_{|\mathbf{m}|-|\mathbf{n}|=N}
   \binom{\mathbf{m}}{\mathbf{n}}^{(d)}\Psi _{\mathbf{m}}^{(d)}(\mathbf{u}) \nonumber \\
   &=
   \sum_{\mathbf{m} \in \mathcal{P}}
   \sum_{\mathbf{n}\subset \mathbf{m}}
   \frac{1}{|\mathbf{m}|-|\mathbf{n}|+1}
   B_{\mathbf{n}}^{(d)}(\mathbf{z})
   \binom{\mathbf{m}}{\mathbf{n}}^{(d)}
   \Psi _{\mathbf{m}}^{(d)}(\mathbf{u}). \nonumber
\end{align}
{\rm{(7)}} By (\ref{eq:key formula 2}) and the summation of a geometric series, 
\begin{align}
\sum_{\mathbf{m} \in \mathcal{P}}
   \sum_{i=0}^{N-1}
   B_{\mathbf{m}}^{(d)}\left(\mathbf{z}+\frac{i}{N}\mathbf{1}\right)
      \Psi _{\mathbf{m}}^{(d)}(\mathbf{u})
   &=
   \sum_{i=0}^{N-1}
   \frac{|\mathbf{u}|}{e^{|\mathbf{u}|}-1}
   {_{0}\mathcal{F}_0}^{(d)}\left( \begin{matrix} \\ \end{matrix};\mathbf{z}+\frac{i}{N}\mathbf{1},\mathbf{u}\right) \nonumber \\
   &=
   \frac{|\mathbf{u}|}{e^{|\mathbf{u}|}-1}
   {_{0}\mathcal{F}_0}^{(d)}\left( \begin{matrix} \\ \end{matrix};\mathbf{z},\mathbf{u}\right)
   \sum_{i=0}^{N-1}
   e^{\frac{i}{N}|\mathbf{u}|} \nonumber \\
   &=
   \frac{|\mathbf{u}|}{e^{|\mathbf{u}|}-1}
   {_{0}\mathcal{F}_0}^{(d)}\left( \begin{matrix} \\ \end{matrix};\mathbf{z},\mathbf{u}\right)
   \frac{e^{|\mathbf{u}|}-1}{e^{\frac{|\mathbf{u}|}{N}}-1} \nonumber \\
   &=
   N\frac{\frac{|\mathbf{u}|}{N}}{e^{\frac{|\mathbf{u}|}{N}}-1}
   {_{0}\mathcal{F}_0}^{(d)}\left( \begin{matrix} \\ \end{matrix};N\mathbf{z},\frac{\mathbf{u}}{N}\right) \nonumber \\
   &=
   N
   \sum_{\mathbf{m} \in \mathcal{P}}
   B_{\mathbf{m}}^{(d)}\left(N\mathbf{z}\right)
      \Psi _{\mathbf{m}}^{(d)}\left(\frac{\mathbf{u}}{N}\right) \nonumber \\
   &=
   \sum_{\mathbf{m} \in \mathcal{P}}
   N^{1-|\mathbf{m}|}B_{\mathbf{m}}^{(d)}\left(N\mathbf{z}\right)
      \Psi _{\mathbf{m}}^{(d)}(\mathbf{u}). \nonumber
\end{align}
{\rm{(8)}} By (\ref{eq:key formula 2})
\begin{align}
\sum_{\mathbf{m} \in \mathcal{P}}
   B_{\mathbf{m}}^{(d)}\left(\mathbf{z}+\mathbf{1}\right)
      \Psi _{\mathbf{m}}^{(d)}(\mathbf{u})
   &=
   \frac{|\mathbf{u}|}{e^{|\mathbf{u}|}-1}
   {_{0}\mathcal{F}_0}^{(d)}\left( \begin{matrix} \\ \end{matrix};\mathbf{z}+\mathbf{1},\mathbf{u}\right) \nonumber \\
   &=
   \frac{|\mathbf{u}|}{e^{|\mathbf{u}|}-1}
   e^{|\mathbf{u}|}
   {_{0}\mathcal{F}_0}^{(d)}\left( \begin{matrix} \\ \end{matrix};\mathbf{z},\mathbf{u}\right) \nonumber \\
   &=
\sum_{\mathbf{n} \in \mathcal{P}}
   B_{\mathbf{n}}^{(d)}\left(\mathbf{z}\right)
   e^{|\mathbf{u}|}\Psi _{\mathbf{n}}^{(d)}(\mathbf{u}) \nonumber \\
   &=
\sum_{\mathbf{n} \in \mathcal{P}}
   B_{\mathbf{n}}^{(d)}\left(\mathbf{z}\right)
   \sum_{\mathbf{m} \in \mathcal{P}}
   \binom{\mathbf{m}}{\mathbf{n}}^{(d)}\Psi _{\mathbf{m}}^{(d)}(\mathbf{u}) \nonumber \\
   &=
\sum_{\mathbf{m} \in \mathcal{P}}
   \sum_{\mathbf{n} \subset \mathbf{m}}
   \binom{\mathbf{m}}{\mathbf{n}}^{(d)}
   B_{\mathbf{n}}^{(d)}\left(\mathbf{z}\right)
   \Psi _{\mathbf{m}}^{(d)}(\mathbf{u}). \nonumber
\end{align}
\end{proof}

\section{A multivariate analogue of the multiple Bernoulli polynomials}
For $n$-tuple complex numbers 
$$
{\boldsymbol{\omega}}
   :=
   (\omega _{1},\ldots , \omega _{n}), \quad \omega _{j} \in \mathbb{C}\setminus \{0\},
$$
we define the multiple Bernoulli polynomials $B_{n, m}\left(z \mid {\boldsymbol{\omega}} \right)$ with a generating function 
\begin{align}
e^{zu}
\prod_{j=1}^{n}
   \frac{u}{e^{\omega _{j}u}-1}
   =
   \sum_{m\geq 0}
      B_{n, m}\left(z \mid {\boldsymbol{\omega}} \right)
      \Psi _{m}(u) \quad (|\omega _{j}u|<2\pi ,j=1,\ldots ,n). 
\end{align}
Let 
\begin{align}
\widehat{{\boldsymbol{\omega}}}(j):=&\,(\omega _{1},\cdots, \omega _{j-1}, \omega _{j+1},\cdots, \omega _{r})\in \mathbb{C}^{r-1} \nonumber \\ 
=&\,(\omega _{1},\cdots,\widehat{\omega}_{j},\cdots, \omega _{r}), \nonumber \\
{\boldsymbol{\omega }}^{-}[j]:=&\,(\omega _{1},\cdots,-\omega _{j},\cdots,\omega _{r})\in \mathbb{C}^{r}. \nonumber 
\end{align}
For $B_{n, m}\left(z \mid {\boldsymbol{\omega}} \right)$, the following formulas are well-known (see \cite{N} (12)--(17)). 
\begin{align}
\label{eq:multiple Bernoulli1}
B_{n,m}(cz\mid c{\boldsymbol{\omega}})=&\,c^{m-n}B_{n,m}(z\mid {\boldsymbol{\omega}}) \,\,\, (c\in \mathbb{C}^{*}), \\
\label{eq:multiple Bernoulli2}
B_{n,m}(\vert {\boldsymbol{\omega}}\vert -z\mid {\boldsymbol{\omega}})=&\,(-1)^{m}B_{n,m}(z\mid {\boldsymbol{\omega}}), \\
\label{eq:multiple Bernoulli3}
B_{n,m}(z+\omega_{j}\mid {\boldsymbol{\omega}})-B_{n,m}(z\mid {\boldsymbol{\omega}})=&\,mB_{n-1,m-1}(z\mid \widehat{{\boldsymbol{\omega}}}(j)), \\
\label{eq:multiple Bernoulli4}
B_{n,m}(z\mid {\boldsymbol{\omega}}^{-}[j])=&\,-B_{n,m}(z+\omega_{j}\mid {\boldsymbol{\omega}}), \\
\label{eq:multiple Bernoulli5}
B_{n,m}(z\mid {\boldsymbol{\omega}})+B_{n,m}(z\mid {\boldsymbol{\omega}}^{-}[j])=&\,-mB_{n-1,m-1}(z\mid \widehat{{\boldsymbol{\omega}}}(j)), \\
\label{eq:multiple Bernoulli6}
\frac{d}{dz}B_{n,m}(z\mid {\boldsymbol{\omega}})=&\,mB_{n,m-1}(z\mid {\boldsymbol{\omega}}). 
\end{align}
We also introduce a multivariate analogue of the multiple Bernoulli polynomials by 
\begin{align}
{_{0}\mathcal{F}_0}^{(d)}\left( \begin{matrix} \\ \end{matrix};\mathbf{z},\mathbf{u}\right)
\prod_{i=1}^{n}
   \frac{|\mathbf{u}|}{e^{\omega _{i}|\mathbf{u}|}-1}
   =
   \sum_{\mathbf{m} \in \mathcal{P}}
      B_{n, \mathbf{m}}^{(d)}\left(\mathbf{z}\mid {\boldsymbol{\omega}} \right)
      \Psi _{\mathbf{m}}^{(d)}(\mathbf{u})
\end{align}
and obtain a multivariate analogue of the above formulas (\ref{eq:multiple Bernoulli1})--(\ref{eq:multiple Bernoulli6}) easily. 
\begin{thm}
{\rm{(1)}} 
\begin{align}
\label{eq:multivariate multiple Bernoulli1}
B_{n,\mathbf{m}}^{(d)}(c\mathbf{z}\mid c{\boldsymbol{\omega}})=&c^{|\mathbf{m}|-n}B_{n,\mathbf{m}}^{(d)}(\mathbf{z}\mid {\boldsymbol{\omega}}) \,\,\, (c\in \mathbb{C}^{*}). 
\end{align}
{\rm{(2)}} 
\begin{align}
\label{eq:multivariate multiple Bernoulli2}
B_{n,\mathbf{m}}^{(d)}(\vert {\boldsymbol{\omega}}\vert\mathbf{1}-\mathbf{z}\mid {\boldsymbol{\omega}})=&(-1)^{|\mathbf{m}|}B_{n,\mathbf{m}}^{(d)}(\mathbf{z}\mid {\boldsymbol{\omega}}).
\end{align}
{\rm{(3)}} 
\begin{align}
\label{eq:multivariate multiple Bernoulli3}
B_{n,\mathbf{m}}^{(d)}(\mathbf{z}+\omega_{j}\mathbf{1}\mid {\boldsymbol{\omega}})-B_{n,\mathbf{m}}^{(d)}(\mathbf{z}\mid {\boldsymbol{\omega}})=&
      \sum_{i=1}^{r}
      B_{n-1,\mathbf{m}_{i}}^{(d)}(z\mid \widehat{{\boldsymbol{\omega}}}(j))
         \left(m_{i}+\frac{d}{2}(r-i)\right)
         h_{-,i}^{(d)}(\mathbf{m}).
\end{align}
{\rm{(4)}} 
\begin{align}
\label{eq:multivariate multiple Bernoulli4}
B_{n,\mathbf{m}}^{(d)}(\mathbf{z}\mid {\boldsymbol{\omega}}^{-}[j])=&-B_{n,\mathbf{m}}^{(d)}(\mathbf{z}+\omega_{j}\mathbf{1}\mid {\boldsymbol{\omega}}). \end{align}
{\rm{(5)}} 
\begin{align}
\label{eq:multivariate multiple Bernoulli5}
B_{n,\mathbf{m}}^{(d)}(\mathbf{z}\mid {\boldsymbol{\omega}})+B_{n,\mathbf{m}}^{(d)}(\mathbf{z}\mid {\boldsymbol{\omega}}^{-}[j])=&
   -\sum_{i=1}^{r}
      B_{n-1,\mathbf{m}_{i}}^{(d)}(z\mid \widehat{{\boldsymbol{\omega}}}(j))
         \left(m_{i}+\frac{d}{2}(r-i)\right)
         h_{-,i}^{(d)}(\mathbf{m}).
\end{align}
{\rm{(6)}} 
\begin{align}
\label{eq:multivariate multiple Bernoulli6}
E_{0}(\mathbf{z})B_{n,\mathbf{m}}^{(d)}(\mathbf{z}\mid {\boldsymbol{\omega}})=&\sum_{i=1}^{r}
      B_{n,\mathbf{m}_{i}}^{(d)}(z\mid {\boldsymbol{\omega}})
         \left(m_{i}+\frac{d}{2}(r-i)\right)
         h_{-,i}^{(d)}(\mathbf{m}).
\end{align}
\end{thm}
\begin{proof}
{\rm{(1)}} From the generating function of the multiple multivariate Bernoulli polynomials and homogeneity of Jack polynomials, we have 
\begin{align}
\sum_{\mathbf{m} \in \mathcal{P}}
   B_{n,\mathbf{m}}^{(d)}(c\mathbf{z}\mid c{\boldsymbol{\omega}})
   \Psi _{\mathbf{m}}^{(d)}(\mathbf{u})
   &=
   {_{0}\mathcal{F}_0}^{(d)}\left( \begin{matrix} \\ \end{matrix};c\mathbf{z},\mathbf{u}\right)
\prod_{i=1}^{n}
   \frac{|\mathbf{u}|}{e^{c\omega _{i}|\mathbf{u}|}-1} \nonumber \\
   &=
   c^{-n}
   {_{0}\mathcal{F}_0}^{(d)}\left( \begin{matrix} \\ \end{matrix};\mathbf{z},c\mathbf{u}\right)
\prod_{i=1}^{n}
   \frac{|c\mathbf{u}|}{e^{\omega _{i}|c\mathbf{u}|}-1} \nonumber \\
   &=
   c^{-n}
   {_{0}\mathcal{F}_0}^{(d)}\left( \begin{matrix} \\ \end{matrix};\mathbf{z},c\mathbf{u}\right)
\prod_{i=1}^{n}
   \frac{|c\mathbf{u}|}{e^{\omega _{i}|c\mathbf{u}|}-1} \nonumber \\
   &=
   c^{-n}
   \sum_{\mathbf{m} \in \mathcal{P}}
      B_{n, \mathbf{m}}^{(d)}\left(\mathbf{z}\mid {\boldsymbol{\omega}} \right)
      \Psi _{\mathbf{m}}^{(d)}(c\mathbf{u}) \nonumber \\
   &=
   \sum_{\mathbf{m} \in \mathcal{P}}
      c^{|\mathbf{m}|-n}
      B_{n, \mathbf{m}}^{(d)}\left(\mathbf{z}\mid {\boldsymbol{\omega}} \right)
      \Psi _{\mathbf{m}}^{(d)}(\mathbf{u}). \nonumber
\end{align}
{\rm{(2)}} By (\ref{eq:key formula 2}), 
\begin{align}
\sum_{\mathbf{m} \in \mathcal{P}}
   B_{n,\mathbf{m}}^{(d)}(\vert {\boldsymbol{\omega}}\vert\mathbf{1}-\mathbf{z}\mid {\boldsymbol{\omega}})
   \Psi _{\mathbf{m}}^{(d)}(\mathbf{u})
   &=
   {_{0}\mathcal{F}_0}^{(d)}\left( \begin{matrix} \\ \end{matrix};\vert {\boldsymbol{\omega}}\vert\mathbf{1}-\mathbf{z},\mathbf{u}\right)
\prod_{i=1}^{n}
   \frac{|\mathbf{u}|}{e^{\omega _{i}|\mathbf{u}|}-1} \nonumber \\
   &=
   e^{\vert {\boldsymbol{\omega}}\vert \vert \mathbf{u}\vert}
   {_{0}\mathcal{F}_0}^{(d)}\left( \begin{matrix} \\ \end{matrix};\mathbf{z},-\mathbf{u}\right)
\prod_{i=1}^{n}
   \frac{|\mathbf{u}|}{e^{\omega _{i}|\mathbf{u}|}-1} \nonumber \\
   &=
   {_{0}\mathcal{F}_0}^{(d)}\left( \begin{matrix} \\ \end{matrix};\mathbf{z},-\mathbf{u}\right)
\prod_{i=1}^{n}
   \frac{|\mathbf{u}|e^{\omega _{i}|\mathbf{u}|}}{e^{\omega _{i}|\mathbf{u}|}-1} \nonumber \\
   &=
   {_{0}\mathcal{F}_0}^{(d)}\left( \begin{matrix} \\ \end{matrix};\mathbf{z},-\mathbf{u}\right)
\prod_{i=1}^{n}
   \frac{|-\mathbf{u}|}{e^{-\omega _{i}|\mathbf{u}|}-1} \nonumber \\
   &=
   \sum_{\mathbf{m} \in \mathcal{P}}
   (-1)^{|\mathbf{m}|}
   B_{n,\mathbf{m}}^{(d)}(\mathbf{z}\mid {\boldsymbol{\omega}})
   \Psi _{\mathbf{m}}^{(d)}(\mathbf{u}). \nonumber
\end{align}
{\rm{(3)}} By (\ref{eq:key formula 2}), (\ref{eq:Pieri 2})
\begin{align}
& \sum_{\mathbf{m} \in \mathcal{P}}
   (B_{n,\mathbf{m}}^{(d)}(\mathbf{z}+\omega_{j}\mathbf{1}\mid {\boldsymbol{\omega}})-B_{n,\mathbf{m}}^{(d)}(\mathbf{z}\mid {\boldsymbol{\omega}}))
   \Psi _{\mathbf{m}}^{(d)}(\mathbf{u}) \nonumber \\
   & \quad =
   \left({_{0}\mathcal{F}_0}^{(d)}\left( \begin{matrix} \\ \end{matrix};\mathbf{z}+\omega_{j}\mathbf{1},\mathbf{u}\right)
   -{_{0}\mathcal{F}_0}^{(d)}\left( \begin{matrix} \\ \end{matrix};\mathbf{z},\mathbf{u}\right)\right)
\prod_{i=1}^{n}
   \frac{|\mathbf{u}|}{e^{\omega _{i}|\mathbf{u}|}-1} \nonumber \\
   & \quad =
   (e^{\omega_{j}|\mathbf{u}|}-1)
   {_{0}\mathcal{F}_0}^{(d)}\left( \begin{matrix} \\ \end{matrix};\mathbf{z},\mathbf{u}\right)
\prod_{i=1}^{n}
   \frac{|\mathbf{u}|}{e^{\omega _{i}|\mathbf{u}|}-1} \nonumber \\
   & \quad =
   |\mathbf{u}|{_{0}\mathcal{F}_0}^{(d)}\left( \begin{matrix} \\ \end{matrix};\mathbf{z},\mathbf{u}\right)
\prod_{1\leq i\not=j\leq n}
   \frac{|\mathbf{u}|}{e^{\omega _{i}|\mathbf{u}|}-1} \nonumber \\
   & \quad =
   \sum_{\mathbf{m} \in \mathcal{P}}
      B_{n-1,\mathbf{m}}^{(d)}(z\mid \widehat{{\boldsymbol{\omega}}}(j))
      |\mathbf{u}|\Psi _{\mathbf{m}}^{(d)}(\mathbf{u}) \nonumber \\
   & \quad =
   \sum_{\mathbf{m} \in \mathcal{P}}
      B_{n-1,\mathbf{m}}^{(d)}(z\mid \widehat{{\boldsymbol{\omega}}}(j))
      \sum_{i=1}^{r}
         \Psi_{\mathbf{m}^{i}}^{(d)}(\mathbf{u})
         \left(m_{i}+1+\frac{d}{2}(r-i)\right)
         h_{-,i}^{(d)}(\mathbf{m}^{i}) \nonumber \\
   & \quad =
   \sum_{\mathbf{m} \in \mathcal{P}}
      \sum_{i=1}^{r}
      B_{n-1,\mathbf{m}_{i}}^{(d)}(z\mid \widehat{{\boldsymbol{\omega}}}(j))
         \left(m_{i}+\frac{d}{2}(r-i)\right)
         h_{-,i}^{(d)}(\mathbf{m})
   \Psi _{\mathbf{m}}^{(d)}(\mathbf{u}). \nonumber
\end{align}
{\rm{(4)}} By (\ref{eq:key formula 2})
\begin{align}
& \sum_{\mathbf{m} \in \mathcal{P}}
   B_{n,\mathbf{m}}(\mathbf{z}\mid {\boldsymbol{\omega}}^{-}[j])
   \Psi _{\mathbf{m}}^{(d)}(\mathbf{u}) \nonumber \\
   & \quad =
   {_{0}\mathcal{F}_0}^{(d)}\left( \begin{matrix} \\ \end{matrix};\mathbf{z},\mathbf{u}\right)
   \frac{|\mathbf{u}|}{e^{-\omega _{j}|\mathbf{u}|}-1}
   \prod_{1\leq i\not=j\leq n}
      \frac{|\mathbf{u}|}{e^{\omega _{i}|\mathbf{u}|}-1} \nonumber \\
   & \quad =
   -e^{\omega _{j}|\mathbf{u}|}{_{0}\mathcal{F}_0}^{(d)}\left( \begin{matrix} \\ \end{matrix};\mathbf{z},\mathbf{u}\right)
   \frac{|\mathbf{u}|}{e^{\omega _{j}|\mathbf{u}|}-1}
   \prod_{1\leq i\not=j\leq n}
      \frac{|\mathbf{u}|}{e^{\omega _{i}|\mathbf{u}|}-1} \nonumber \\
   & \quad =
   -{_{0}\mathcal{F}_0}^{(d)}\left( \begin{matrix} \\ \end{matrix};\mathbf{z}+\omega _{j}\mathbf{1},\mathbf{u}\right)
   \prod_{i=1}^{n}
      \frac{|\mathbf{u}|}{e^{\omega _{i}|\mathbf{u}|}-1} \nonumber \\
   & \quad =
   \sum_{\mathbf{m} \in \mathcal{P}}
   -B_{n,\mathbf{m}}(\mathbf{z}+\omega _{j}\mathbf{1}\mid {\boldsymbol{\omega}})
   \Psi _{\mathbf{m}}^{(d)}(\mathbf{u}). \nonumber     
\end{align}
{\rm{(5)}} By (\ref{eq:multivariate multiple Bernoulli4}) and (\ref{eq:multivariate multiple Bernoulli3}), we have
\begin{align}
& B_{n,\mathbf{m}}(\mathbf{z}\mid {\boldsymbol{\omega}})+B_{n,\mathbf{m}}(\mathbf{z}\mid {\boldsymbol{\omega}}^{-}[j]) \nonumber \\
   & \quad =
   B_{n,\mathbf{m}}(\mathbf{z}\mid {\boldsymbol{\omega}})-B_{n,\mathbf{m}}(\mathbf{z}+\omega _{j}\mathbf{1}\mid {\boldsymbol{\omega}}) \nonumber \\
   & \quad =
      -\sum_{i=1}^{r}
      B_{n-1,\mathbf{m}_{i}}^{(d)}(z\mid \widehat{{\boldsymbol{\omega}}}(j))
         \left(m_{i}+\frac{d}{2}(r-i)\right)
         h_{-,i}^{(d)}(\mathbf{m}). \nonumber
\end{align}
{\rm{(6)}} By (\ref{eq:key formula 3}) and (\ref{eq:Pieri 2}), we have 
\begin{align}
\sum_{\mathbf{m} \in \mathcal{P}}
   E_{0}(\mathbf{z})B_{n,\mathbf{m}}^{(d)}(\mathbf{z}\mid {\boldsymbol{\omega}})
   \Psi _{\mathbf{m}}^{(d)}(\mathbf{u})
   &=
   E_{0}(\mathbf{z}){_{0}\mathcal{F}_0}^{(d)}\left( \begin{matrix} \\ \end{matrix};\mathbf{z},\mathbf{u}\right)
\prod_{i=1}^{n}
   \frac{|\mathbf{u}|}{e^{c\omega _{i}|\mathbf{u}|}-1} \nonumber \\
   &=
   |\mathbf{u}|{_{0}\mathcal{F}_0}^{(d)}\left( \begin{matrix} \\ \end{matrix};\mathbf{z},\mathbf{u}\right)
\prod_{i=1}^{n}
   \frac{|\mathbf{u}|}{e^{c\omega _{i}|\mathbf{u}|}-1} \nonumber \\
   &=
\sum_{\mathbf{m} \in \mathcal{P}}
   B_{n,\mathbf{m}}^{(d)}(\mathbf{z}\mid {\boldsymbol{\omega}})
   |\mathbf{u}|\Psi _{\mathbf{m}}^{(d)}(\mathbf{u}) \nonumber \\
   &=
\sum_{\mathbf{m} \in \mathcal{P}}
   B_{n,\mathbf{m}}^{(d)}(\mathbf{z}\mid {\boldsymbol{\omega}})
      \sum_{i=1}^{r}
         \Psi_{\mathbf{m}^{i}}^{(d)}(\mathbf{u})
         \left(m_{i}+1+\frac{d}{2}(r-i)\right)
         h_{-,i}^{(d)}(\mathbf{m}^{i}) \nonumber \\
   &=\sum_{i=1}^{r}
      B_{n,\mathbf{m}_{i}}^{(d)}(z\mid {\boldsymbol{\omega}})
         \left(m_{i}+\frac{d}{2}(r-i)\right)
         h_{-,i}^{(d)}(\mathbf{m})
         \Psi_{\mathbf{m}}^{(d)}(\mathbf{u}). \nonumber
\end{align}
\end{proof}

\section{Concluding remarks}
Since our multivariate Bernoulli polynomials have various properties which are regarded as a natural generalization of (\ref{eq:found prop1}) - (\ref{eq:found prop7}), our multivariate Bernoulli polynomials are regarded as a good multivariate analogue of Bernoulli polynomials. 
% Therefore we desire to find a multivariate zeta function associated with our multivariate Bernoulli polynomials, that is constructing zeta functions whose some special values are written by $B_{\mathbf{m}}^{(d)}\left(\mathbf{z}\right)$. 
Therefore we desire to find a multivariate zeta function whose some special values are written by our multivariate Bernoulli polynomials $B_{\mathbf{m}}^{(d)}\left(\mathbf{z}\right)$. 

\section*{Acknowledgements}
We thank Professor M. Noumi for his precious advices on Jack and interpolation Jack polynomials.

%%%%%%%%%%%%%%%

\bibliographystyle{amsplain}

\noindent 
Department of Mathematics, Graduate School of Science, Kobe University, \\
1-1, Rokkodai, Nada-ku, Kobe, 657-8501, JAPAN\\
% \noindent Department of Pure and Applied Mathematics, 
% Graduate School of Information Science and Technology, Osaka University, \\
% 1-5, Yamadaoka, Suita, Osaka, 565-0871, JAPAN\\
% 1-1, Machikaneyama, Toyonaka, Osaka 560-0043, JAPAN.\\
E-mail: g-shibukawa@math.kobe-u.ac.jp

\end{document}